\documentclass{amsproc}

\newtheorem{Th}{Theorem}[section]

\newtheorem{Lemma}[Th]{Lemma}
\theoremstyle{definition}
\newtheorem{Def}[Th]{Definition}

\theoremstyle{remark}

\newtheorem{Obs}{Remark}

\newcommand{\lra}{\longrightarrow}
\newcommand{\C}{\mathbb{C}}
\newcommand{\N}{\mathbb{N}}
\newcommand{\Po}{\mathcal{P}}
\newcommand{\rd}{\mathrm{d}}

\begin{document}
\title[Orthogonally additive polynomials in $\ell_p$.]{On the representation
of orthogonally additive polynomials in $\ell_p$.}
\author{Alberto Ibort}
\address{Departamento de Matem{\'a}ticas\\ Universidad Carlos III de Madrid\\
Avda. de la Universidad 30\\  28911 Legan{\'e}s\\  Spain}
\email{albertoi@math.uc3m.es}
\thanks{The first author was supported in part by Project MTM 2004-07090-C03}
\author{Pablo Linares}
\address{Departamento de An{\'a}lisis Matem{\'a}tico\\
Facultad de Matem{\'a}ticas\\ Universidad Complutense de Madrid\\
28040 Madrid\\ Spain.}
\email{plinares@mat.ucm.es}
\thanks{The second author was partially supported by the "Programa de formaci{\'o}n
del profesorado universitario del MEC".}

\author{Jos{\'e} G. Llavona}
\address{Departamento de An{\'a}lisis Matem{\'a}tico\\
Facultad de Matem{\'a}ticas\\ Universidad Complutense de Madrid\\
28040 Madrid\\  Spain.}
\email{JL\_Llavona@mat.ucm.es}
\thanks{The second and third author were supported in part by Project
MTM 2006-03531}





\begin{abstract}
We present a new proof of a  Sundaresan's result which shows that
the space of orthogonally additive polynomials $\Po_o(^k\ell_p)$
is isometrically isomorphic to $\ell_{p/p-k}$ if $k<p<\infty$ and
to $\ell_\infty$ if $1\leq p\leq k$.
\end{abstract}
\maketitle

\section{Introduction}

A continuous scalar-valued map $P$ in a Banach space is called a
$k$-homogeneous polynomial if there exists a continuous $k$-linear
form $\phi$ on $X$ such that $P(x)=\phi(x,\dots,x)$. The space of
$k$-homogeneous polynomials on $X$ will be denoted by $\Po(^kX)$.
It is a Banach space with the norm
    $$\|P\|=\sup_{\|x\|\leq 1}|P(x)|$$

We will denote by $\widehat{\bigotimes}_{\pi,k}X$ and by
$\widehat{\bigotimes}_{\pi,s,k}X$ the completed $k$-fold
projective tensor product and the completed $k$-fold projective
symmetric tensor product res\-pec\-ti\-ve\-ly, where $\pi$ denotes
the projective norm. For $X=\ell_p$, the tensor diagonal is
defined to be the closed subspace of
$\widehat{\bigotimes}_{\pi,s,k}X$ generated by
$e_n\otimes\overbrace{\cdots}^k\otimes e_n$. It will be denoted by
$D_{k,p}$.

It will be needed the well known result that the dual of the space
$\widehat{\bigotimes}_{\pi,s,k}X$ is isometrically isomorphic to
the space $\Po(^kX)$. For notations and results about homogeneous
polynomials, the reader is referred to \cite{Dineen} or
\cite{Mujica}, and for the theory of tensor products to
\cite{Ryan}.

We are interested in the subspace of $\Po(^kX)$ consisting of all
orthogonally additive polynomials. Recall that if $X$ is a Banach
lattice, a $k$-homogeneous polynomial $P$ is said to be
orthogonally additive if $P(x+y)=P(x)+P(y)$ whenever $x$ and $y$
are orthogonal or disjoint elements of $X$ (that is
$|x|\wedge|y|=0$). We will consider $X=\ell_p$ as a Banach lattice
with its natural order given by $x=(x_n)\leq y=(y_n)$ if $x_n\leq
y_n$. See \cite{Lind-Tzafriri} for more information about the
theory of Banach lattices.

We give here a new proof of a result of Sundaresan (see
\cite{Sundaresan}). This result has been recently generalized to
all Banach lattices by Benyamini, Lasalle and Llavona in
\cite{BLLl}. There are also independent proofs for the case of
$X=C(K)$, see \cite{CLZ} and \cite{PerVill}. The reason to present
this new proof is that, in our opinion, it is much simpler in the
sense that there is no need of hard tools such us the
representation of orthogonally additive functionals using
Caratheodory functions, as in \cite{Sundaresan}, or the Kakutani
representation theorem for Banach lattices used in \cite{BLLl}.
This new proof also makes more evident the underlying ideas.

\section{The tensor diagonal}

In this section we generalize Example 2.23 in \cite{Ryan} to give
a description of the tensor diagonal in
$\widehat{\bigotimes}_{\pi,s,k}\ell_p$. The main technique was a
Rademacher averaging (\cite{Ryan}, Lema 2.22). Its generalization
requires the $k$-Rademacher generalized functions introduced by
Aron and Globevnik in \cite{AG}:

\begin{Def}[\cite{AG}]\label{RadGen}
Fix $k\in \N$, $k\geq 2$ and let
$\alpha_1=1,\alpha_2,\dots,\alpha_k$ denote the $n^{th}$ roots of
unity. Let $r_1:[0,1]\lra \C$ be the step function taking the
value $\alpha_j$ on $(j-1/k,j/k)$ for $j=1,\dots,n$. Assuming that
$r_{n-1}$ has been defined, define $r_n$ in the following way: fix
any of the $k_{n-1}$ sub-intervals $I$ of $[0,1]$ used in the
definition of $r_{n-1}$. Divide $I$ into $k$ equal intervals
$I_1,\dots,I_k$ and set $r_n(t)=\alpha_j$ if $t\in I_j$.
\end{Def}

We will also need the following

\begin{Lemma}[\cite{AG}]\label{LemaAG}
For each $k=2,3,\dots$ the associated functions $r_n$ satisfy the
following properties:

\begin{itemize}
\item $|r_n(t)|=1$ for all $n\in\N$ and all $t\in[0,1]$.
\item For any choice of $n_1,\dots,n_k$

\begin{displaymath}
\int_0^1 r_{n_1}(t)\cdots r_{n_k}(t)= \left\{ \begin{array}{l}
1  \textrm{ if $n_1=\dots=n_k$} \\
0 \textrm{ otherwise.}
\end{array}\right.
\end{displaymath}
\end{itemize}
\end{Lemma}

The next lemma is a generalization of Lemma 2.22 in \cite{Ryan}:

\begin{Lemma}[Rademacher Averaging]\label{RadAver}
Let $X_1,\dots X_k$ vector spaces and let $x_{1,1},\dots,x_{1,n}
\in X_1,\dots,$ and $x_{k,1},\dots,x_{k,n} \in X_k$. Then
    $$\sum_{i=1}^nx_{1,i}\otimes\dots\otimes x_{k,i}=\int_0^1
    \left(\sum_{i=1}^nr_i(t)x_{1,i}\right)\otimes\cdots
    \otimes\left(\sum_{i=1}^nr_i(t)x_{k,i}\right)\rd t$$
\end{Lemma}

\begin{proof}
Just expand the integral and use the second property of Lemma
\ref{LemaAG}.
\end{proof}

\begin{Th}\label{Tensordiagonal}
Let $1\leq p < \infty$. The tensor diagonal $D_{k,p}$ in
$\widehat{\bigotimes}_{\pi,s,k}\ell_p$ is iso\-me\-tri\-ca\-lly
isomorphic to $\ell_{p/k}$ if $k<p<\infty$ and to $\ell_{1}$ if
$1\leq p\leq k$.
\end{Th}

\begin{proof} $\ $\\

\begin{enumerate}
\item $k<p<\infty$. \\
Let $u=\sum_{i=0}^n a_i e_i\otimes\dots\otimes e_i\in D_{k,p}$.
Using Lemma \ref{RadAver}, we write
    $$u=\int_0^1
    \left(\sum_{i=1}^n sign(a_i)|a_i|^{1/k}r_i(t)e_i\right)\otimes\cdots
    \otimes\left(\sum_{i=1}^n |a_i|^{1/k}r_i(t)e_{i}\right)\rd t$$
and, like in \cite{Ryan} (pages 34-35) we get:

\begin{eqnarray}
\pi(u) &\leq & \sup_{0\leq t\leq 1}
    \left\|\sum_{i=1}^n sign(a_i)|a_i|^{1/k}r_i(t)e_i \right\|_p
    \cdots \left\|\sum_{i=1}^n |a_i|^{1/k}r_i(t)e_{i}\right\|_p={}
                        \nonumber\\
    & &= \left( \sum_{i=1}^{n}|a_i|^{p/k}\right)^{k/p}=\|(a_i)\|_{p/k}
                        \nonumber
\end{eqnarray}

To prove the identity, define a $k$-linear form on $\ell_p$ by
$B(x_1,\dots,x_n)=\sum b_i x_{1,i}\cdots x_{k,i}$ where
$x_j=(x_{j,n})$ and $b_i=sign(a_i)|a_i|^{p/k-1}$. Using H{\"o}lder's
inequality, it is easy to see that $\|B\|\leq
(\sum_{i=1}^n|a_i|^{p/k})^{1-p/k}$ and then
    $$\sum_{i=1}^n|a_i|^{p/k}=|\langle u,B\rangle|\leq \pi(u)(\sum_{i=1}^n|a_i|^{p/k})^{1-p/k}.$$

Hence $\|(a_i)\|_{p/k}\leq \pi(u)$ and therefore $D_{k,p}$ is
isometrically isomorphic to $\ell_{p/k}$.

\item $1\leq p\leq k$. \\
Let be $u=\sum_{i=0}^n a_i e_i\otimes\dots\otimes e_i\in D_{k,p}$,
then $\pi(u)\leq \sum_{i=0}^n |a_i|$. Re\-ci\-pro\-ca\-lly, define
$B(x_1,\dots,x_n)=\sum_{i=0}^n sign(a_i)x_{1,i}\cdots x_{k,i}$, we
have that $|B(x_1,\dots,x_n)|\leq \|x_1\|_k\cdots\|x_k\|_k\leq
\|x_1\|_p\cdots\|x_k\|_p$ and so $\|B\|\leq 1$. Then
    $$\pi(u)\geq \langle u,B\rangle=\sum_{i=1}^\infty |a_i|$$
and we are done.
\end{enumerate}
\end{proof}

\begin{Obs}
Definition \ref{RadGen} gives the classical Rademacher functions
for the case $k=2$ and these results are those in \cite{Ryan},
(Example 2.23, page 34).
\end{Obs}
\section{The main result}

Our proof is based in the fact that the orthogonally additive
polynomials are isometrically isomorphic to the dual of the tensor
diagonal $D_{k,p}$. We need a previous lemma:

\begin{Lemma}
Let $1\leq p<\infty$. The dual of the tensor diagonal $D_{k,p}^*$
is iso\-me\-tri\-cally isomorphic to $\ell_\infty$ if $1\leq p\leq
k$ and to $\ell_{p/p-k}$ if $k<p<\infty$ in the sense that for
every $F\in D_{k,p}^*$, $(F(e_i\otimes\dots\otimes e_i))$ is in
$\ell_\infty$ for the first case and in $\ell_{p/p-k}$ for the
second.
\end{Lemma}

\begin{proof}
The proof is standard, observe that $\ell_{p/p-k}$ is the dual of
$\ell_{p/k}$ and carry on the same proof of $\ell_q^*=\ell_{q'}$
for $1/q+1/q'=1$ with the identification of the projective norm
shown in Theorem \ref{Tensordiagonal}.
\end{proof}

We are now ready to prove the theorem:

\begin{Th}\label{MainT}
Let $1\leq p<\infty$. The space of orthogonally additive,
k-ho\-mo\-ge\-neous polynomials $\Po_o(^k\ell_p)$ is isometrically
isomorphic to $\ell_\infty$ for $1\leq p\leq k$ and to
$\ell_{p/p-k}$ for $k<p<\infty$.
\end{Th}

\begin{proof}
The proof consists on showing that $\Po_o(^k\ell_p)$ is
isometrically isomorphic to $D_{k,p}^*$. We will suppose that
$k<p<\infty$. The other case is analogous.

Let $F\in D_{k,p}^*$, the correspondence is established by
associating $F$ to the polynomial
$P(x)=\widetilde{F}(x\otimes\dots\otimes x)$ where $\widetilde{F}
\in (\widehat{\bigotimes}_{\pi,k,s}\ell_p)^*$ is defined by
\begin{displaymath}
\widetilde{F}(e_{n_1}\otimes\dots\otimes e_{n_k})= \left\{
\begin{array}{l}
F(e_{n_1}\otimes\dots\otimes e_{n_k})  \textrm{ if $n_1=\dots=n_k$} \\
0 \textrm{ otherwise}
\end{array}\right.
\end{displaymath}

To see that $\widetilde{F}$ is well defined let $x=\sum x_ne_n\in
\ell_p$, then

   \begin{eqnarray}
|\widetilde{F}(x\otimes\dots\otimes x)| &\leq & \sum
|x_{n_1}\cdots x_{n_k} \widetilde{F}(e_{n_1}\otimes\dots\otimes
e_{n_k})|= \sum | x_n^k F(e_{n}\otimes\dots\otimes e_{n})| {}
                        \nonumber\\
    & & \leq \|(x_n^k)\|_{p/k}\|(F(e_{n}\otimes\dots\otimes
    e_{n}))\|_{p/p-k}\leq
    \|F\|\|x\|_p^k={}
                        \nonumber\\
    & & =\|F\|\pi(x\otimes\dots\otimes x)
                        \nonumber
\end{eqnarray}
hence $\widetilde{F}$ is continuous and $\|\widetilde{F}\|\leq
\|F\|$, then we have the equality since $\widetilde{F}$ was an
extension of $F$.

To see that $P(x)=\widetilde{F}(x\otimes\dots\otimes x)$ is
orthogonally additive, note that is enough to check that the
$k$-linear symmetric form $\phi$ associated to $P$, verifies that
$\phi(e_{n_1},\dots, e_{n_k})$ is zero whenever at least two of
its entries will be different and this is true by definition of
$\widetilde{F}$.

Finally as the correspondence between polynomials and the dual of
the sym\-me\-tric tensor product is an isometric isomorphism,
$\|P\|=\|\widetilde{F}\|=\|F\|$ which completes the proof.
\end{proof}

As it has been shown, if $P(x)=\phi(x,\dots,x)$ is an orthogonally
additive polynomial on $\ell_p$, the essential information of $P$
is contained in the sequence $(\phi(e_n,\dots,e_n))$. There is
another possible proof of Theorem \ref{MainT} using a result by
Zalduendo (see Corollary 1 of \cite{Zalduendo}):

\begin{Lemma}[\cite{Zalduendo}]
Let $k<p$ and let $\phi$ a continuous $k$-linear form on $\ell_p$.
Then
    $$(\phi(e_n,\dots,e_n))\in \ell_{p/p-k}.$$
\end{Lemma}

From this result, it can be shown as well that the space of
orthogonally additive polynomials is isometrically isomorphic to
$\ell_{p/p-k}$ if $k<p$. We don't repeat the proof since the ideas
are essentially the same, however the proof presented here is
self-contained.

\begin{Obs}
This method will also be valid for $1\leq p\leq k$ since in this
case, trivially $(\phi(e_n,\dots,e_n))\in \ell_\infty$. It is
shown in \cite{Zalduendo} that this is the best characterization.
\end{Obs}


\end{document}